\documentclass[submission,copyright,creativecommons]{eptcs}
\providecommand{\event}{AFL 2023} 

\usepackage{iftex}
\usepackage{amsmath}
\usepackage{mleftright}
\usepackage{amsthm}
\usepackage{graphicx}

\newtheorem{definition}{Definition}
\newtheorem{proposition}{Proposition}

\ifpdf
  \usepackage{underscore}         
  \usepackage[T1]{fontenc}        
\else
  \usepackage{breakurl}           
\fi

\title{A Construction for Variable Dimension\\
  Strong Non-Overlapping Matrices}
\author{Elena Barcucci
\institute{University of Florence\\ Italy}
\email{elena.barcucci@unifi.it}
\and
Antonio Bernini
\institute{University of Florence\\ Italy}
\email{antonio.bernini@unifi.it}
\and
Stefano Bilotta
\institute{University of Florence\\ Italy}
\email{stefano.bilotta@unifi.it}
\and
Renzo Pinzani
\institute{University of Florence\\ Italy}
\email{renzo.pinzani@unifi.it}
}
\def\titlerunning{Strong Non-Overlapping Matrices}
\def\authorrunning{E. Barcucci, A. Bernini, S. Bilotta \& R. Pinzani}
\begin{document}
\maketitle

\begin{abstract}
We propose a method for the construction of sets of variable dimension 
strong non-overlapping matrices basing on any strong non-overlapping set of strings.
\end{abstract}

\section{Introduction}

Intuitively, two matrices do not overlap if it is not possible 
to move one over the other in a way such that the 
corresponding entries match. In some recent works (\cite{BBBP1},\cite{BBBP2},\cite{BBBP3}) the 
matrices are constructed by imposing some constraints on their rows which must avoid 
some particular consecutive patterns or must have some fixed entries in particular 
positions. The matrices of the sets there defined have the same fixed dimension.

In the present paper, we deal with matrices having different dimensions and we construct them by means a different approach: we move from any strong non-overlapping set $W$ of strings, defined over a finite alphabet, and, in a very few words, the strings of $W$ becomes the rows of our matrices. The method is general and once the cardinality of the strings of $W$ with a same length is known, the cardinality of the set of matrices is straightforward.

\medskip
This work could fit in the theory of bidimensional codes, as well as non overlapping 
sets of strings do in the theory of codes. Moreover, if the latter have been used in 
telecommunication systems both theory and engineering \cite{BS,WW}, the matrices of our sets could be 
useful in the field of digital image processing, and a possible (future) application of 
this kind of sets is in the template matching which is a technique to discover if small 
parts of an image match a template image.

\section{Preliminaries}
Let $\mathcal M_{m\times n}$ be the set of all the matrices with $m$ rows and $n$ columns. Given a matrix $A\in \mathcal M_{m\times n}$, we consider a block partition
\begin{equation}\label{pb}
A =(A_{i,j})= \begin{bmatrix}
A_{11} & \ldots &A_{1k}\\
\vdots & \ldots &\vdots\\ 
A_{h1} & \ldots & A_{hk}
\end{bmatrix}\ .
\end{equation}

Let us define $fr(A_{ij})$ the \emph{frame} of a block $A_{ij}$ of $A$. Intuitively, it is a set tracking the borders of the block which lie on the top ($t$), left ($l$), right ($r$) and bottom ($b$) border of the matrix $A$. More precisely, the set $fr(A_{i,j})$ is a subset of $\{t,b,l,r\}$ defined as follows:
\begin{definition}\label{frame}
$$
fr(A_{i,j})\supseteq
\begin{cases}
t, &\text{if } i=1\\
b, &\text{if } i=h\\
l, &\text{if } j=1\\
r, &\text{if } j=k\\
\end{cases}
\ \ \ \ \ .
$$
\end{definition}

\noindent
For example, if
$
A=
\begin{bmatrix}
A_{11}&A_{12}&A_{13}
\end{bmatrix}
$
($h=1$ and $k=3$)
then $fr(A_{11})=\{t,b,l\}$,  $fr(A_{12})=\{t,b\}$, and $fr(A_{13})=\{t,b,r\}$ since $i=h=1$. But if 

$$
A=
\begin{bmatrix}
A_{11}&A_{12}&A_{13}\\
A_{21}&A_{22}&A_{23}\\
A_{31}&A_{32}&A_{33}
\end{bmatrix}
$$
then $fr(A_{11})=\{t,l\}$,  $fr(A_{12})=\{t\}$,  
$fr(A_{13})=\{t,r\}$ and similarly for the other blocks. Note that 
in this case $fr(A_{22})=\emptyset$.

\bigskip
\begin{definition}\label{sovrapponibili}
Given two matrices $A\in \mathcal M_{m\times n}$ and $B\in \mathcal M_{m'\times n'}$, they are said
\emph{overlapping} if there exist two suitable block partitions
$A=(A_{ij})$
, 
$B=
(B_{i'j'})
$, and some $i,j,i',j'$
such that
\begin{itemize}
\item $A_{i,j}=B_{i'j'}$, and
\item $fr(A_{ij})\cup fr(B_{i'j'})=\{t,l,r,b\}$. 
\end{itemize}

\noindent
In the case $A=B$, the matrix is said \emph{self-overlapping}.
\end{definition}

\bigskip\noindent
To illustrate the definition, the following examples are given:
\begin{itemize}
\item
Given the two matrices
$$
A=\mleft[
\begin{array}{cc|cc|c}
  1&2&1&1&2 \\
  0&1&0&3&0 \\
  \hline
  3&2&1&0&2 \\
  0&1&3&1&3 \\
\end{array}
\mright]
\quad  \mbox{and} \quad 
B=\mleft[
\begin{array}{cc}
  2&1\\
  \hline
  1&1\\
  0&3\\
\end{array}
\mright]\ ,
$$
they overlap since the entries of the blocks $A_{12}$ and $B_{21}$ coincide. Moreover, we have $fr(A_{12})=\{t\}$, $fr(B_{21})=\{l,b,r\}$ so that $fr(A_{12})\cup fr(B_{21})=\{l,t,b,r\}$.

\item If
$
B=\mleft[
\begin{array}{c|cc}
  3&1&2\\
  2&0&1\\
\end{array}
\mright]
$
the matrix $A$ (as before) and the matrix $B$ again overlap since $A_{11}=B_{12}$ and $fr(A_{11})\cup fr(B_{12})=\{l,r,b,t\}$ being $fr(A_{11})=\{l,t\}$ and $fr(B_{12})=\{t,r,b\}$.
 
\item Note that if
$
B=\mleft[
\begin{array}{cc|c}
  1&2&3\\
  0&1&2\\
\end{array}
\mright]
$, even if 
$A_{11}=B_{11}$, we have $fr(A_{11})=\{l,t\}$ and 
$fr(B_{11})=\{l,t,b\}$ so that $fr(A_{11})\cup 
fr(B_{11})=\{l,b,t\}\neq \{l,t,b,r\}$. Nevertheless, the two matrices are overlapping 
since, considering the block partitions
$
B=
\mleft[
\begin{array}{cc}
  B_{11}&B_{12}\\
  B_{21}&B_{22}\\
\end{array}
\mright]
=\mleft[
\begin{array}{c|cc}
  1&2&3\\
  \hline
  0&1&2\\
\end{array}
\mright]
$
and
$$
A=
\begin{bmatrix}
A_{11}&A_{12}\\
A_{21}&A_{22}\\
\end{bmatrix}
=
\mleft[
\begin{array}{cc|ccc}
  1&2&1&1&2 \\
  \hline
  0&1&0&3&0 \\
  3&2&1&0&2 \\
  0&1&3&1&3 \\
\end{array}
\mright]\ ,
$$
we have $A_{11}=B_{22}$ and
$fr(A_{11})\cup 
fr(B_{22})=\{l,t,b,r\}$.

\item As a further example we consider the particular case where $A=[A_{11}]$ and 
$
B=
\begin{bmatrix}
B_{11}&B_{12}&B_{13}\\
B_{21}&B_{22}&B_{23}\\
B_{31}&B_{32}&B_{33}
\end{bmatrix}
$
 with $B_{22}=A_{11}$. Here, we have $fr({A_{11}})\cup fr(B_{22})=
 \{t,b,l,r\}\cup\{\emptyset\}=\{t,b,l,r\}$ and the two matrices are overlapping.
\item We conclude this list of examples showing two matrices $A$ and $B$ such that, even if they have two equal blocks ($A_{11}=B_{11}$), they are not overlapping since the second condition on the frames of the blocks of Definition \ref{sovrapponibili} is not fulfilled $(\mbox{since}\ fr({A_{11}})\cup fr(B_{11})=
 \{t,l\}\neq\{t,b,l,r\})$:
$$
A=\mleft[
\begin{array}{cc|cc|c}
  1&2&1&1&2 \\
  0&1&0&3&0 \\
  \hline
  3&2&1&0&2 \\
  0&1&3&1&3 \\
\end{array}
\mright]
\quad  \mbox{,} \quad 
B=\mleft[
\begin{array}{cc|c}
  1&2&3\\
  0&1&1\\
  \hline
  1&0&3\\
\end{array}
\mright]\ .
$$
\end{itemize}

From these examples, it should be clear that if two matrices are overlapping, then the common block naturally induces a block partition $(A_{i,j})$ for $A$ (and a block partition $(B_{i,j})$) such that the number of blocks in each its row and column can be not larger than $3$. Figure \ref{fig1} shows two examples of the least fine block partitions for two overlapping matrices $A$ and $B$ induced by the (gray) common block. Therefore, the block partitions \ref{pb} involved in Definition \ref{sovrapponibili} are such that $h,k\in\{1,2,3\}$.

\begin{center}
\begin{figure}[!h]
\centering
\includegraphics[trim={3cm 7cm 0cm 3cm},clip,scale=.5]{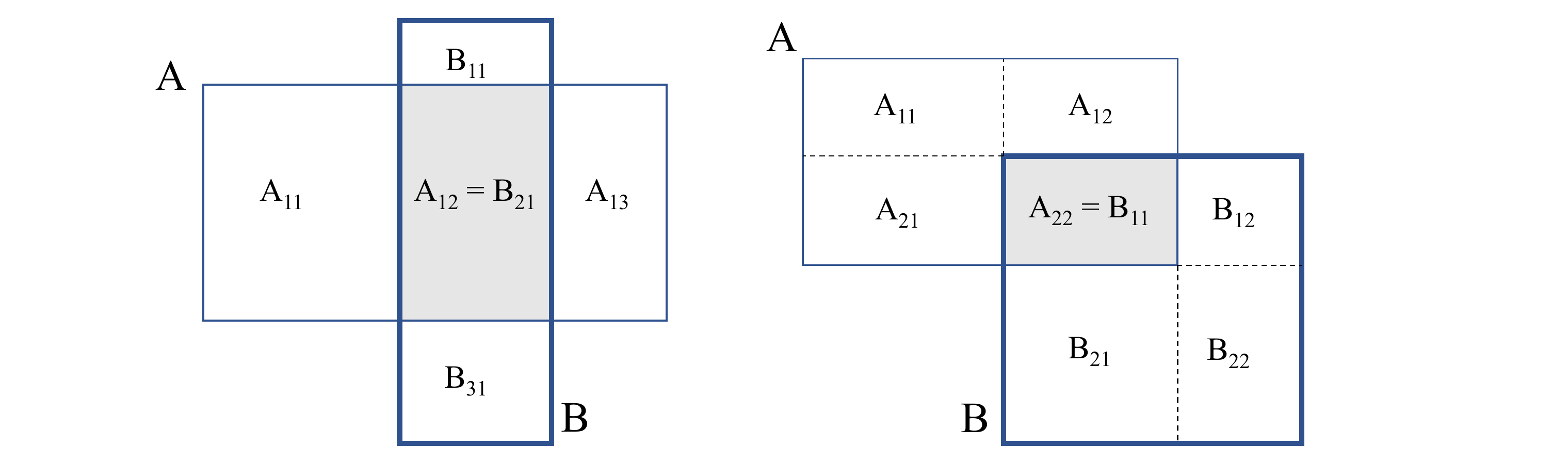}
\caption{The least fine block partition in two examples of two overlapping matrices}
\label{fig1}
\end{figure}    
\end{center}

\medskip
\noindent
We note that if a matrix is completely contained in the other, then the two matrices are overlapping according to Definition \ref{sovrapponibili}, as in the second to last example of the above list.
In the context of strings, the scenario is different, as illustrated in the following. Two strings are said \emph{overlapping} if there is a proper prefix of one that is equal to a proper suffix of the other. Consequently, they are said to be \emph{non-overlapping} if there is no a proper prefix of one that is equal to a proper prefix of the other (these definitions are more formally recalled, later in this section). It can happen that, given two non-overlapping strings, one of them is an inner factor of the other, as in the case of the two binary strings $1111000$ and $10$. If this is not allowed, then the strings are said \emph{strong non-overlapping} (i.e. two strings are strong non-overlapping if they are non-overlapping and if one of them is not an inner factor of the other), as in the case of the two binary strings $1111000$ and $10100$. In short,  being non-overlapping strings or strong non-overlapping strings are different concepts.

In our framework, if two matrices $A$ and $B$ are not overlapping then it can not happen that one of them (say $B$) is completely contained in the other. Indeed, if this were the case, then the smaller matrix $B$ could be trivially partitioned in one block $B=B_{11}$ so that $fr(B_{11})=\{t,b,l,r\})$. Moreover, it would be $B_{11}=A_{ij}$ for some block $A_{ij}$ of the matrix $A$, and the matrices $A$ and $B$ would be overlapping, whatever the block $A_{i,j}$. 

Therefore, when two matrices are not overlapping, we prefer to call them strong non-overlapping matrices (instead of simply non-overlapping matrices), in order to emphasize that certainly neither is contained in the other. Then, we give the following formal definitions characterizing two such matrices and a set of strong non-overlapping matrices:

\begin{definition}
The matrices $A$ and $B$ are said \emph{strong non-overlapping} if there does
not exist any block partition for $A$ and $B$, and any $i,j,i',j'$ 
such that $A_{i,j}=B_{i',j'}$ or, if such block partitions exist, then 
$fr(A_{ij})\cup fr(B_{i'j'})\neq\{t,l,r,b\}$.
\end{definition}






\bigskip
\begin{definition}
A set $\mathcal P$ of matrices is said to be \emph{strong non-overlapping} if each matrix is self non-overlapping and if for any two matrices in $\mathcal P$ they are strong non-overlapping. 
\end{definition}

\bigskip
\noindent
For completeness, let us recall some notions about non-overlapping and strong non-overlapping sets of strings.

Given a finite alphabet $\Sigma$, a string $v\in \Sigma^*$ is said to be \emph{self non-overlapping} (often said \emph{unbordered} or 
equivalently \emph{bifix-free}) if any proper prefix of $v$ is different from 
any proper suffix of $v$ (for more details see \cite{BP}).  

Two self non-overlapping strings $v$, $v'\in \Sigma^*$ are said to be 
\emph{non-overlapping} (or equivalently \emph{cross bifix-free}) if any proper prefix of 
$v$ is different from any 
proper suffix of $v'$, and vice versa. A set of strings is said to be a 
\emph{non-overlapping set} (or \emph{cross bifix-free set}) of strings if each element of 
the set is slef non-overlapping and if any two strings are non-overlapping. 

\begin{definition}
Two non-overlapping strings $v$ and $v'$ are said to be \emph{strong non-overlapping} if 
there do not exist $\alpha,\beta \in \Sigma^*$, with $\alpha$ and $\beta$ not both empty, 
such that $v'=\alpha v\beta$ (or $v=\alpha v'\beta$).
\end{definition}

In other words, the strong non-overlapping property requires that the shortest string between $v$ and $v'$ (if any) does not occur as an inner factor
in the other one (\cite{BBP2,B}).
For example, if $v=1100$ and $v'=1\mathbf{1100}100$, then $v$ and 
$v'$ are non-overlapping but they are not strong non-overlapping since $v'$ contains an 
occurrence of $v$ (in bold).

\begin{definition}
A set of strings is said to be a \emph{strong non-overlapping set} if any two strings of 
the set are strong non-overlapping.
\end{definition}

\section{Construction of the set of matrices}
Let $\mathcal V_n=\displaystyle\bigcup_{s\leq n} V^s$ be a variable dimension 
strong non-overlapping set of strings where each $V^s$ is a non-overlapping set of 
strings of length $s$, for $s_0 \leq s \leq n$, where $s_0 \geq 2$ is the minimum string length.
We now define a set of variable dimension matrices, using strings of a same length $s$ of $V^{s}$ as rows of a matrix. In the following, the two matrices $C$ and $D$ of dimension $m_1\times s$ and $m_2\times t$, respectively, are constructed with the rows $C_i^s\in V^s$ and $D_j^t\in V^t$, with $i=1,2,\ldots,m_1$ and $j=1,2,\ldots, m_2$. 
$$
C=
\left(
\begin{matrix}
C_1^{s}\\
C_2^{s}\\
\vdots\\
\vdots\\
C_{m_1}^{s}
\end{matrix}
\right)
\hspace{1cm}
D=
\left(
\begin{matrix}
D_1^{t}\\
D_2^{t}\\
\vdots\\
D_{m_2}^{t}
\end{matrix}
\right)
$$  

It is not difficult to show that if $C$ and $D$ have a different 
number of columns (then $s\neq t$) they can not be overlapping (see next proposition).

Unfortunately, in the case $C$ and $D$ have the same number of 
columns ($s=t$), then the two matrices can present a ``vertical" 
overlap. More precisely:
\begin{itemize}
\item the matrix $D$ could be equal to a sub-matrix of $C$ constituted by $m_2$ consecutive rows of $C$ (or vice versa):

$$
C=
\begin{bmatrix}
C_{11}\\
C_{12}\\
C_{13}
\end{bmatrix}
=
\begin{bmatrix}
C_{11}\\
D\\
C_{13}
\end{bmatrix}
$$

\noindent
(with either blocks $C_{11}$ or $C_{13}$ possibly empty).

\item the first (last) $\ell$ rows of $D$ could be equal to the last (first) $\ell$ rows of $C$ (or vice versa):

$$
C=
\begin{bmatrix}
C_{11}\\
C_{12}
\end{bmatrix}
=
\mleft[
\begin{array}{c}
  C_{11}   \vspace{.1cm}\\
  \hline  \vspace{-.3cm}\\
  D_1^t \vspace{.07cm}\\
  D_2^t \vspace{.07cm}\\
  \vdots\\
  D_{\ell}^t
\end{array}
\mright]
\hspace{1cm}
D=
\begin{bmatrix}
D_{11}\\
D_{12}\\
\end{bmatrix}
=
\mleft[
\begin{array}{c}
  D_1^t\\
  D_2^t\\
  \vdots\\
  D_{\ell}^t 
  \vspace{.1cm}\\
  \hline  \vspace{-.3cm}\\
  D_{12}
\end{array}
\mright]\ .
$$
\end{itemize}
In order to avoid the situations described above, we introduce a constraint for the first and the last row of each matrix: all the matrices with the same number $s$ of columns must have the same first row $T^s\in V^s$ and the same last row $B^s\in V^s$, with $T^s\neq B^s$. Also, these two selected rows cannot appear as inner rows of any other matrix with that number $s$ of columns. In other words, we force:
\begin{itemize}
	\item the top row $T^s$ of all the matrices with the same number $s$ of columns to be  the same;
	\item the bottom row $B^s$ of all the matrices with the same number $s$ of columns to be  the same;
	\item $T^s\neq B^s$; 
	\item the rows $T^s$ and $B^s$ not to occur in any other line of the matrix.
\end{itemize} 

\noindent
Formally, the matrices $C$ with the same number $s$ of columns  must have the following structures:

$$
C=
\left(
\begin{matrix}
T^{s}\\
\\
C_2^{s}\\
\vdots\\
\vdots\\
C_{m_1-1}^{s}\\
\\
B^{s}
\end{matrix}
\right)
$$

\noindent
with $C_j^{s}\neq T^{s},B^{s}$,
for
$j= 2,3,\ldots, m_1-1$, and $C_j^{s},T^{s},B^{s}\in V^{s}$.

\bigskip
\noindent
We can now define the set $\mathcal V_{m\times n}^{(\leq)}$ of variable-dimension matrices as follows:

\begin{definition}\label{construction}
	Let $\mathcal V_n=\displaystyle\bigcup_{s\leq n} V^s$ be a variable dimension 
strong non-overlapping set of strings where each $V^s$ is a non-overlapping set of 
strings of length $s$, for $s_0 \leq s \leq n$, where $s_0 \geq 2$ is the minimum string length. Moreover,
	let
	
	$$
	\mathcal V_{m \times n}^{(\leq)}=
	\bigcup
	M
	$$
	be the union of the matrices $M$ where $M\in \mathcal M_{h\times s}$, with $2 \leq h\leq m$ 
	and $s_0 \leq s\leq n$, such that
	$$
	M=
	\left\{
	\begin{pmatrix}
	T^{s}\\
	A_2^{s}\\
	\vdots\\
	A_{h-1}^{s}\\
	B^{s}		
	\end{pmatrix}
	\right\}
	$$

\medskip\noindent
	with
$	A_j^{s},T^{s},B^{s}\in V^{s}
	\ \mbox{and}\ 
	A_j^{s}\neq T^{s},B^{s}\ \mbox{for}\ 
	j=2,3,\ldots, h-1\ .
$
\end{definition}

\medskip
The matrices $M\in \mathcal V_{m\times n}^{(\leq)}$ have at most $m$ rows and $n$ columns.
They are constructed by means of $h\leq m$ strings of length $s\leq n$ belonging to 
$\mathcal V_n$.
All the matrices $M$ with the same number $s$ of columns have the same bottom row $B_s$ and the same top row $T_s$, which are not the same. Moreover, each inner row is different from $T_s$ and $B_s$.

\medskip
 
We have the following proposition:

\begin{proposition}\label{prop}
The set $\mathcal V_{m\times n}^{(\leq)}$ is a strong non-overlapping set of variable-dimension matrices.
\end{proposition}

\begin{proof} Let $C,D\in \mathcal V_{m\times n}^{(\leq)}$ and suppose that $C$ and $D$ are two overlapping matrices: then there exists a block matrix $E\in\mathcal M_{r\times c}$ such that $E=C_{i,j}=D_{i',j'}$ fore some two blocks $C_{i,j}$ and $D_{i'j'}$ in two suitable block partitions of $C$ and $D$, and with $fr(C_{i,j})\cup fr(D_{i'j'})=\{l,t,r,b\}$. We have

$$
E=
\mleft(
\begin{matrix}
e_{11} & \ldots &e_{1c}\\
\vdots & \ldots &\vdots\\ 
e_{r1} & \ldots & e_{rc}
\end{matrix}
\mright)\ \ .
$$

For each row $e_{\ell}$, with $\ell=1,2,\ldots,r$, there exist two rows $C_i, D_j\in \mathcal V_n$ such that one of the following cases occurs:
\begin{itemize}
\item $C_i=ue_{\ell}v$ and $D_j=e_{\ell}$, with either $u$ or $v$ possibly empty, where $u,v\in\Sigma^*$;
\item $C_i=ue_{\ell}$ and $D_j=e_{\ell}v$;
\item $C_i=e_{\ell}v$ and $D_j=ue_{\ell}$.
\end{itemize}

\noindent
In any case, the strings $C_i$ and $D_j$ are not strong non-overlapping strings (since they overlap over $e_{\ell}$) against the hypothesis $C_i,D_j\in \mathcal V_{n}$
. 
\end{proof}

\bigskip
We note that in the case $\mathcal V_n$ is a variable dimension non-overlapping set of strings (i.e. the non-overlapping property is not required to be strong), the resulting matrices are not strong non-overlapping according to Definition \ref{sovrapponibili}, since it is possible that one of the two matrices is completely contained in the other one as a suitable block. If we did not contemplate this possibility in Definition \ref{sovrapponibili}, then two matrices constructed with such a $\mathcal V_n$ could be considered still non-overlapping (according to a different definition of non-overlapping matrices). 

\medskip
Moreover, if $\mathcal V_n$ contains strings all of the same lengths, then Proposition \ref{prop} still holds: the matrices will have all the same number of columns.

\medskip
Finally, if $|V^s|$ denotes the cardinality of the non-overlapping set $V^s$, it is straightforward to deduce the following formula for the cardinality of $\mathcal V_{m\times n}^{(\leq)}$:
\begin{equation}\label{card}
|\mathcal V_{m\times n}^{(\leq)}|=\sum_{h\leq m}\ \sum_{s\leq n}(|V^s|-2)^{h-2}\ .
\end{equation}

The two terms $-2$ in the above formula take into account that the first and the last row in the matrices with $s$ columns are fixed and can not occur as inner rows.

\bigskip

For the sake of clearness, we propose an example for the construction of a set of variable dimension strong non-overlapping matrices.
Let $V^3= \{ 110, 210, 310, 320 \}$ and $V^5=\{22000,23000,33000\}$ be two  sets of non-overlapping strings over the alphabet $\Sigma=\{0,1,2,3\}$. It is easily seen that $V^3 \cup V^5$ is a strong non-overlapping code. Then, we construct 

$$\mathcal V_{4\times 5}^{(\leq)}= \mathcal M_{2\times 3}^{(\leq)} \cup \mathcal M_{3\times 3}^{(\leq)} \cup \mathcal M_{4\times 3}^{(\leq)} \cup \mathcal M_{2\times 5}^{(\leq)} \cup \mathcal M_{3\times 5}^{(\leq)} \cup \mathcal M_{4\times 5}^{(\leq)}$$ where:

$$
\mathcal M_{2\times 3}=
\left\{
\left(\begin{matrix}
1&1&0\\
3&2&0
\end{matrix}
\right)
\right\}
$$

$$
\mathcal M_{3\times 3}=
\left\{
\left(\begin{matrix}
1&1&0\\
2&1&0\\
3&2&0
\end{matrix}
\right)
,
\left(\begin{matrix}
1&1&0\\
3&1&0\\
3&2&0
\end{matrix}
\right)
\right\}
$$

$$
\mathcal M_{4\times 3}=
\left\{
\left(\begin{matrix}
1&1&0\\
2&1&0\\
2&1&0\\
3&2&0
\end{matrix}
\right)
,
\left(\begin{matrix}
1&1&0\\
2&1&0\\
3&1&0\\
3&2&0
\end{matrix}
\right)
,
\left(\begin{matrix}
1&1&0\\
3&1&0\\
2&1&0\\
3&2&0
\end{matrix}
\right)
,
\left(\begin{matrix}
1&1&0\\
3&1&0\\
3&1&0\\
3&2&0
\end{matrix}
\right)
\right\}
$$

$$
\mathcal M_{2\times 5}=
\left\{
\left(\begin{matrix}
2&2&0&0&0\\
3&3&0&0&0
\end{matrix}
\right)
\right\}
$$

$$
\mathcal M_{3\times 5}=
\left\{
\left(\begin{matrix}
2&2&0&0&0\\
2&3&0&0&0\\
3&3&0&0&0
\end{matrix}
\right)
\right\}
$$

$$
\mathcal M_{4\times 5}=
\left\{
\left(\begin{matrix}
2&2&0&0&0\\
2&3&0&0&0\\
2&3&0&0&0\\
3&3&0&0&0
\end{matrix}
\right)
\right\}
$$

The reader can easily check that $\mathcal V_{4\times 5}^{(\leq)}$ is a set of variable dimension strong non-overlapping matrices having cardinality 10 according to (\ref{card}).

\section{Conclusions}

The paper provides a simple and general method to generate a set of strong 
non-overlapping matrices over a finite alphabet, once a strong non-overlapping set of 
strings (over the same alphabet) is at our disposal. The crucial point is the constraint 
on the first and last rows which must be the same for all the matrices with the same 
number of columns.

Using the variable length strong non-overlapping sets of strings defined in \cite{B} and 
\cite{BBP2}, two different set of strong non-overlapping matrices arise which could be 
compared in terms of cardinality or its asymptotic behaviour.

\bigskip
Moreover, the construction we proposed, in the case of fixed dimension matrices, gives the possibility
to list them in a Gray code sense, following the studies started in \cite{BBBP1,BBP1,BBPSV2,BBPSV1,BBPV2,BBPV1} where different Gray codes are defined for several set of strings and matrices.

In this case, we generate the matrices moving from a set of non-overlapping strings $V^s$ of length $s$ and we suppose that there exists a Gray code $GV^s$ for $V^s$:
$$
GV^s=\{w_1,w_2.\ldots,w_t,w_{t+1},w_{t+2}\} \text{ with } t>0\ .
$$

\noindent
Note that we require $|V^s|\geq3$. We choose two strings from $GV^s$. Without loss of generality, we choose $w_{t+1}$ and $w_{t+2}$ and we define the set of matrices 
$M_{h+2,s}$ with $h+2$ rows and $s$ columns where the first and last rows are, respectively, the strings $w_{t+1}$ and $w_{t+2}$:
\vspace{-.3cm}
	$$
	M_{h+2,s}=
	\left\{
	\begin{pmatrix}
	w_{t+1}\\
	C_1^{s}\\
	\vdots\\
	C_h^{s}\\
	w_{t+2}		
	\end{pmatrix}
	\Bigg\vert\ C_i^s\in V^s\setminus\{w_{t+1},w_{t+2}\}
	\right\}
	\ .
	$$
\noindent 
Let $N_{h,s}$ be the set of matrices obtained by $M_{h+2,s}$ removing the first and last rows:
\vspace{.1cm}
$$
	N_{h,s}=
	\left\{
	\begin{pmatrix}
	C_1^{s}\\
	\vdots\\
	C_h^{s}\\
	\end{pmatrix}
	\Bigg\vert\ C_i^s\in V^s\setminus\{w_{t+1},w_{t+2}\}
	\right\}
	\ .
	$$

\bigskip
\noindent
Clearly, the cardinality of $N_{h,s}$ and $M_{h+2,s}$ is the same and denoting it by $q$ it is $q=t^h$.

We now recursively define a Gray code $GN_{h,s}$ for the set $N_{h,s}$. If $h=1$, then the list $GN_{1,s}=(w_1),(w_2),\ldots,(w_t)$ is a Gray code (since it is obtained by $GV^s$ where the strings are read as matrices of dimension $1\times s$). Suppose now that $GN_{h,s}=A_1,A_2,\ldots,A_q$ is a Gray code where $h\geq1$ and $A_i\in N_{h,s}$, for $i=1,2,\ldots,q$. The following list $GN_{h+1,s} $of matrices, defined as block matrices,

$$
GN_{h+1,s}=
\mleft[
\begin{array}{c}
  w_1   \vspace{.1cm}\\
  \hline  \vspace{-.3cm}\\
  A_1\\
\end{array}
\mright]
\cdots
\mleft[
\begin{array}{c}
  w_1   \vspace{.1cm}\\
  \hline  \vspace{-.3cm}\\
  A_q\\
\end{array}
\mright]
\mleft[
\begin{array}{c}
  w_2   \vspace{.1cm}\\
  \hline  \vspace{-.3cm}\\
  A_q\\
\end{array}
\mright]
\cdots
\mleft[
\begin{array}{c}
  w_2   \vspace{.1cm}\\
  \hline  \vspace{-.3cm}\\
  A_1\\
\end{array}
\mright]
\cdots\cdots
\mleft[
\begin{array}{c}
  w_t   \vspace{.1cm}\\
  \hline  \vspace{-.3cm}\\
  A_{\ell}\\
\end{array}
\mright]
\cdots
\mleft[
\begin{array}{c}
  w_t   \vspace{.1cm}\\
  \hline  \vspace{-.3cm}\\
  A_{q+1-\ell}\\
\end{array}
\mright]\ \ ,
$$

\bigskip
\noindent
where
$$
\ell=
\begin{cases}
q, &\text{if } t \text{ is even}\\
1, &\text{if } t \text{ is odd}
\end{cases}
\ \ ,
$$

\medskip
\noindent
is easily seen to be a Gray code since the lists $A_1,A_2,\ldots,A_q$ and 
$w_1,w_2,\ldots,w_t$ are Gray codes for hypothesis.

Finally, adding the strings $w_{t+1}$ and $w_{t+2}$, respectively, as first and last rows to all the $q$ matrices $A_1,A_2,\ldots,A_q$ of $GN_{h,s}$ we obtain a Gray code $GM_{h+2,s}$ for the set $M_{h+2,s}$:
$$
GM_{h+2,s}=
\mleft[
\begin{array}{c}
  w_{t+1}   \vspace{.1cm}\\
  \hline  \vspace{-.3cm}\\
  A_1\\
  \hline  \vspace{-.3cm}\\
  w_{t+2}
\end{array}
\mright]
\cdots\cdots\cdots
\mleft[
\begin{array}{c}
  w_{t+1}   \vspace{.1cm}\\
  \hline  \vspace{-.3cm}\\
  A_q\\
  \hline  \vspace{-.3cm}\\
  w_{t+2}
\end{array}
\mright]\ \ .
$$

\bibliographystyle{eptcs}
\bibliography{generic}
\end{document}